\documentclass[a4paper]{article}
\usepackage[usenames,dvipsnames]{xcolor}
\usepackage{amsthm,amssymb,amsmath,enumerate,graphicx}
\usepackage{tikz}
\usepackage[british]{babel}
\usepackage{mathtools}
\usepackage{todonotes}

\usepackage{bbm} % Fuer indicator function..
\theoremstyle{plain}

\newtheorem{definition}{Definition}
\newtheorem{proposition}[definition]{Proposition}
\newtheorem{theorem}[definition]{Theorem}

\newtheorem{lemma}[definition]{Lemma}
\newtheorem{conjecture}[definition]{Conjecture}

\newtheorem{claim}[definition]{Claim}

\newcommand{\chif}{\chi'_f}
\newcommand{\sm}{\setminus}

\newcommand{\comment}[1]{}

\newcommand{\Z}{\mathbb Z}

\renewcommand{\subset}{\subseteq}

\newcommand{\pos}[1]{[#1]^+}
\newcommand{\Tsub}[2]{T_{(#1,#2)}}

\newcommand{\newl}{d}

\tikzstyle{hvertex}=[thick,circle,inner sep=0.cm, minimum size=2.5mm, fill=white, draw=black]
\tikzstyle{hedge}=[very thick]
\tikzstyle{medge}=[ultra thick,black,densely dotted, preaction={draw,line width=2pt,white}]
\colorlet{hellgrau}{black!30!white}
\colorlet{dunkelgrau}{black!50!white}

\newcommand{\mycolour}{Peach}

\usetikzlibrary{calc}
\usetikzlibrary{backgrounds}

\renewcommand{\c}[2]{\textcolor{blue}{#2}\textcolor{red}{#1}}

\title{Chromatic index, treewidth and maximum degree}
\author{Henning Bruhn, Laura Gellert and Richard Lang\thanks{The research leading to these results was partially supported by EPSRC, grant no. EP/P002420/1.}}
\date{}

\begin{document}
\maketitle

\begin{abstract}
 We conjecture that any graph $G$ with treewidth~$k$ and maximum degree $\Delta(G)\geq k + \sqrt{k}$ satisfies $\chi'(G)=\Delta(G)$. In support of the conjecture we prove its fractional version.
We also show that any graph $G$ with treewidth~$k\geq 4$ and maximum degree $2k-1$ satisfies $\chi'(G)=\Delta(G)$, extending an old result of Vizing.
\end{abstract}
%ACM classification G.2.2: Graph Theory

% MSC classification 05C15( coloring of graphs and hypergraphs), 05C72(Fractional graph theory, fuzzy graph theory),  	05C75  (Structural characterization of families of graphs)
\section{Introduction}

The least number $\chi'(G)$ of colours necessary to properly  colour the edges of a (simple) graph $G$ is either the maximum degree $\Delta(G)$ or $\Delta(G)+1$.
But 
to decide whether $\Delta(G)$ or $\Delta(G)+1$ colours suffice is a difficult
algorithmic problem~\cite{Hol81}. 

Often, graphs with a relatively simple structure can be edge-coloured with only $\Delta(G)$
colours. This is the case for bipartite graphs (K\"onig's theorem) and for cubic Hamiltonian graphs.
%TRIVIAL! ~\cite[Ex. 5.6]{Wilson}. 
Arguably, one measure of simplicity is \emph{treewidth}, how closely a 
graph resembles a tree. (See next section for a definition.)

Vizing~\cite{Vizing_russ} (see also Zhou et al.~\cite{Zhou_Na_Ni96}) observed a 
consequence of his adjacency lemma: any graph with treewidth~$k$ and maximum degree at least~$2k$ has chromatic index $\chi'(G)=\Delta(G)$.\footnote{{More generally, the same holds for $k$-degenerate graphs (see Section~\ref{Sec:degenerate}).}}
Is this tight? No, it turns out.
Using two recent adjacency lemmas we can decrease the required maximum degree: 
\begin{proposition}\label{prop:edge-colouring} For any graph $G$ of treewidth $k\geq 4$ and maximum degree $\Delta(G)\geq 2k-1$ it holds that $\chi'(G)=\Delta(G)$.
\end{proposition}
This immediately suggests the question: 
how much further can the maximum degree be lowered? 
We conjecture:

\begin{conjecture} \label{conjecture}
Any graph of treewidth $k$ and maximum degree $\Delta\geq k + \sqrt{k}$ has chromatic index $\Delta$.
\end{conjecture}

The bound is close to best possible: in Section~\ref{Sec:discussion} we construct, 
for infinitely many $k$, graphs with treewidth~$k$, maximum degree~$\Delta=k+\lfloor\sqrt k\rfloor< k+\sqrt k$, and chromatic index $\Delta+1$. For other values~$k$ the conjecture (if true) might be off by~$1$ from 
the best bound on $\Delta$. This is, for instance, the case for $k=2$, where the conjecture 
is known to hold. Indeed, Juvan et al.~\cite{JMT99}
show that series-parallel graphs with maximum degree $\Delta\geq 3$ are even $\Delta$-edge-choosable.

In support of the conjecture we prove  its fractional version:
\begin{theorem}\label{thm:fractional-tw-maxdeg}
Any simple graph of treewidth $k$ and maximum degree $\Delta\geq k + \sqrt{k}$ has fractional chromatic index $\Delta$.
\end{theorem}

The theorem follows from a new upper bound on the number of edges:
	\begin{equation*} 
		2 |E(G)| \leq \Delta |V(G)| - (\Delta - k)(\Delta - k +1)
	\end{equation*}
The bound is proved in Proposition~\ref{prop:edge_bound}. It implies quite directly
that no graph with treewidth~$k$ and maximum degree $\Delta\geq k+\sqrt k$
can be overfull. (A graph $G$ is \emph{overfull} if it has an odd number $n$ of vertices
and strictly more than $\Delta(G)\tfrac{n-1}{2}$ edges; a subgraph $H$ of $G$ is an
\emph{overfull subgraph}
if it is overfull and satisfies $\Delta(H)=\Delta(G)$.)

Thus, for certain parameters our conjecture coincides with the overfull conjecture of 
Chetwynd and Hilton~\cite{Che_Hil86}:
\newtheorem*{ofconj}{Overfull conjecture}
\begin{ofconj}
Every graph $G$ on less than $3\Delta(G)$ vertices can be edge-coloured with $\Delta(G)$
colours unless it contains an overfull subgraph.
\end{ofconj}
Because we can exclude that graphs with treewidth $k$ and maximum degree $\Delta\geq k+\sqrt k$
 are overfull, the overfull conjecture (as well as our conjecture) implies that such graphs 
on less than~$3\Delta$ vertices can always be edge-coloured with $\Delta$ colours. 

\medskip
Graphs of treewidth~$k$ are in particular $k$-degenerate (see Section~\ref{Sec:definitions} for the definition of treewidth and Section~\ref{Sec:degenerate} for a discussion on degenerate graphs). Indeed, Vizing~\cite{Vizing_russ} originally showed 
that $k$-degenerate graphs, rather than treewidth~$k$ graphs, of maximum degree $\Delta \geq 2k$ have an
edge-colouring with $\Delta$ colours.
We briefly list some related work  on edge-colourings and their variants
in $k$-degenerate graphs.
Isobe et al.~\cite{I_Zhou_Ni07} show that any $k$-degenerate graph of maximum degree~$\Delta \geq 4k+3$
has a total colouring with only~$\Delta+1$ colours. 
For graphs that are not only $k$-degenerate but also of treewidth~$k$, 
 a maximum degree of $\Delta \geq 3k-3$ already suffices~\cite{BLS16}. 
Noting that they are $5$-degenerate, we include some results on planar graphs as well. 
Borodin, Kostochka and Woodall~\cite{Bo_Ko_Woo97,Bo_Ko_Woo97_1} showed that planar graphs have list-chromatic index $\Delta(G)$ and total chromatic number $\chi''(G)=\Delta(G)+1$ if $\Delta(G) \geq 11$ or if the maximum degree and the girth are at least $5$.
Vizing~\cite{Vizing_russ} proved that a planar graph $G$ has a $\Delta(G)$-edge-colouring if $\Delta(G) \geq 8$. Sanders and Zhao~\cite{SanZhao01} and independently Zhang~\cite{Zha00} extended this to $\Delta(G) \geq 7$.

\nocite{SST+12}

\section{Definitions} \label{Sec:definitions}

All graphs in this article are finite and simple.
 We use standard graph theory notation as found in the book of Diestel~\cite{Die}.

 For a graph $G$ a \emph{tree-decomposition} $(T,\mathcal{B})$ consists
of a tree $T$ and a collection  $\mathcal{B} = \{B_t \text{ : } t \in
V(T) \}$ of \emph{bags} $B_t \subset V(G)$ such that
 \begin{enumerate}[(i)]
  \item $V(G) = \bigcup\limits_{t \in V(T)} B_t,$
  \item for each edge $vw \in E(G)$ there exists a vertex $t \in V(T)$ such that
$v,$ $w \in B_t$, and
  \item if $v \in B_{s} \cap B_{t},$ then $v \in B_r$ for each vertex
$r$  on the path connecting $s$ and $t$ in $T$.
\end{enumerate}
A tree-decomposition $(T,\mathcal{B})$ has \emph{width} $k$ if
each bag has a size of at most $k+1$. The \emph{treewidth} of $G$ is
the smallest integer $k$ for which there is a width $k$
tree-decomposition of $G$. \\
A tree-decomposition  $(T,\mathcal{B})$ of
width $k$ is \emph{smooth} if
\begin{enumerate}[(i)]
\setcounter{enumi}{3}
  \item  $ |B_t| = k+1$ for all $t \in V(T)$ and
  \item  $|B_s \cap B_t| = k$ for all $st \in E(T)$.
 \end{enumerate}
All tree decompositions considered in this paper will be smooth. 
This is possible as a graph of treewidth at most $k$ always has a smooth tree-decomposition of width~$k$; see Lemma 8 in Bodlaender~\cite{Bodlaender98}. 

The fractional chromatic index of a graph $G$ is defined as
\begin{equation*}
 \chi'_f(G) 
= \min 
\left\lbrace
\sum_{M \in \mathcal{M}}  \lambda_M~:~  
	\lambda_M \in \mathbb{R_+},  
	\sum_{M \in \mathcal{M}}  \lambda_M \mathbbm{1}_{ M}(e) = {1}~\quad \forall e \in E(G)
\right\rbrace,
\end{equation*}
where $\mathcal{M}$ denotes the collection of all matchings in $G$ and $\mathbbm{1}_{ M}$  the characteristic vector of $M$. 
For more details on the fractional chromatic index, see for instance Scheinerman and Ullman~\cite{SUBook}.

\section{A bound on the number of edges}

Theorem~\ref{thm:fractional-tw-maxdeg} follows quickly from a bound on the number of edges:

\begin{proposition} \label{prop:edge_bound}
A graph $G$ of treewidth $k$ and maximum degree $\Delta(G) \geq k$ satisfies 

	\begin{equation}\label{eq:edge_bound} 
		2 |E(G)| \leq \Delta(G) |V(G)| - (\Delta(G) - k)(\Delta(G) - k +1).
	\end{equation}

\end{proposition}

Before proving Proposition~\ref{prop:edge_bound} we present one of its consequences:

\begin{lemma}\label{lem:not_overfull}
	Let $G$ be a graph of treewidth at most~$k$ and maximum degree $\Delta\geq k+\sqrt k$. Then $G$ is not overfull.
\end{lemma}
\begin{proof}
	Proposition~\ref{prop:edge_bound} implies
	\[
	\frac{2|E(G)|}{|V(G)|-1} \leq \frac{\Delta |V(G)| -(\Delta-k)(\Delta-k+1)}{|V(G)|-1}
	=\frac{\Delta |V(G)| -(\Delta-k)^2-\Delta+k}{|V(G)|-1}
	\]
	and as $\Delta\geq k+\sqrt k$ we obtain
	\[
	\frac{2|E(G)|}{|V(G)|-1}\leq 
	\frac{\Delta |V(G)| -k-\Delta+k}{|V(G)|-1}=\Delta.
	\]
	This finishes the proof.
\end{proof}

It follows from Edmonds' matching polytope theorem that $\chi'_f(G)= \Delta(G)$, if the 
graph~$G$ does not contain any 
overfull subgraph of maximum degree $\Delta$; see~\cite[Ch. 28.5]{Schri}. 
As the treewidth of a subgraph is never larger than the treewidth of the original graph,
Theorem~\ref{thm:fractional-tw-maxdeg} is a  consequence of Lemma~\ref{lem:not_overfull}.\\

The proof of Proposition~\ref{prop:edge_bound} rests on two lemmas.
We defer their proofs to the end of the section. For a tree $T$ we write $|T|$ to denote the number of its vertices.
If $st\in E(T)$ is an edge of $T$ then we let $\Tsub{s}{t}$ be 
the component of $T-st$ containing $s$. For any number $k$
we set $\pos{k}=\max(k,0)$.

\begin{lemma}\label{lem:baumlemma-1} 
	For a tree $T$ and a positive integer $\newl \leq |T|$ it holds that
	\begin{equation*} 
    	\sum_{(s,t):st\in E(T)} \pos{\newl-|\Tsub{s}{t}|} \geq \newl(\newl-1).
	\end{equation*}
\end{lemma}

If $T^*$ is a subtree of $T$ then let  $\delta^+(T^*)$
be the set of ordered pairs $(s,t)$ so that $st$ is an edge of $T$ with 
$s\in V(T^*)$ but $t\notin V(T^*)$.
(That is, $\delta^+(T^*)$ may be seen as the set of oriented edges leaving $T^*$.)

\begin{lemma}\label{lem:baumlemma-2}
	Let $T$ be a tree and let $\newl \leq |T|$ be a positive integer.
	Then for any subtree $T^* \subset T$ it holds that
 		\begin{equation}\label{equ:baumlemma-2}
    		\sum_{(s,t)\in\delta^+(T^*)} \pos{\newl-|\Tsub{s}{t}|}
			\leq 
			\pos{\newl -|T^*|}.
		 \end{equation}
\end{lemma}

We introduce one more piece of notation.
If $(T,\mathcal{B})$ is a tree decomposition of the graph $G$, then 
for any vertex $v$ of $G$
we denote by $T(v)$ the subtree of $T$ that consists of those vertices corresponding to bags
that contain~$v$.

\begin{proof}[Proof of Proposition~\ref{prop:edge_bound}]
Let $(T,\mathcal{B})$ be a smooth tree decomposition of $G$ of width~$k$.
First note that for any vertex $v$ of $G$, the number of vertices
in  the union of all bags 
containing $v$ is at most~$|T(v)|+k$ since the 
tree decomposition is smooth. 
Thus  $\deg(v)\leq |T(v)|+k-1 $. 

Set $\newl =  \Delta -k +1 \geq 1$, and observe that $d\leq |V(G)|-k= |T|$
as the tree decomposition is smooth.
We calculate
	\begin{align*}
		\Delta - \deg(v) &\geq \pos{\Delta - k + 1 -|T(v)|} \\
		&= 
		\pos{\newl -|T(v)|}
   		\geq
		\sum_{(s,t)\in\delta^+(T(v))} \pos{\newl-|\Tsub{s}{t}|},	
	\end{align*}
where the last inequality follows from Lemma~\ref{lem:baumlemma-2}.

Consider an edge $st\in E(T)$. Since the tree decomposition is smooth
there is  exactly one vertex $v \in V(G)$ with $v \in
B_s$ and $v \notin B_t$. Setting
 $\phi((s,t)) = v$ then defines a function
from the set of all $(s,t)$ with $st\in E(T)$ into $V(G)$.
Note that $\phi((s,t))=v$ if and only if $(s,t)\in\delta^+(T(v))$.
Summing the previous inequality over all vertices, we get
  \begin{align*}
   \sum_{v \in V(G)}(\Delta - \deg(v))
   &\geq  
   \sum_{v \in V(G)} \sum_{(s,t)\in\phi^{-1}(v)} \pos{\newl-|\Tsub{s}{t}|}\\
   &=
   \sum_{(s,t): st\in E(T)} \pos{\newl-|\Tsub{s}{t}|}
   \geq
	\newl(\newl-1),
\end{align*}
where the last inequality is due to Lemma~\ref{lem:baumlemma-1}. This directly implies~\eqref{eq:edge_bound}.
  \end{proof}

It remains to prove Lemma~\ref{lem:baumlemma-1} and~\ref{lem:baumlemma-2}.

\begin{proof}[Proof of Lemma~\ref{lem:baumlemma-1}]
We proceed by induction on $|T| -d$.
The induction starts  when $\newl=|T|$. 
Then $\pos{d-|\Tsub{s}{t}|}=d-|\Tsub{s}{t}|$ and thus
\begin{align*}
\sum_{(s,t): st\in E(T)} \pos{\newl-|\Tsub{s}{t}|} 
&=
\sum_{st \in E(T)} \left( |T|-|\Tsub{s}{t}|  
+ |T|-|\Tsub{t}{s}| \right)\\
	&=
	\sum_{st \in E(T)} |\Tsub{t}{s}| + |\Tsub{s}{t}|
	=
	(|T|-1)|T|.
\end{align*}
	
	Now, let $\newl \leq |T| -1$, which implies in particular $|T|\geq 2$.
Then $T$ has a leaf~$\ell$. 
We set $T'= T-\ell$ and 
	note that $\newl \leq  |T| -1=|T'|$.

Observe that for any edge $st\in E(T')$ we get
    \begin{equation*}
    	|\Tsub{s}{t}|  
    	= 
    	\begin{cases}    
    		|\Tsub{s}{t}'|+1     &\text{ if } \ell\in V(\Tsub{s}{t}),\\
       		|\Tsub{s}{t}'|       &\text{ if } \ell\notin V(\Tsub{s}{t}).
        \end{cases}
    \end{equation*}
We denote by $F$ the set of all $(s,t)$ 
for which $st$ is an edge in $T'$ with $\ell\in V(\Tsub{s}{t})$ and 
with  $|\Tsub{s}{t}'| \leq \newl-1$.
Then
\begin{equation} \label{equ:[d-T'st]^+}
\pos{\newl-|\Tsub{s}{t}|}= 
\begin{cases}
\pos{\newl-|\Tsub{s}{t}'|} -1 
& \text{ if }(s,t)\in F, \\
\pos{\newl-|\Tsub{s}{t}'|}   
&\text{ if }(s,t)\notin F.
\end{cases}
\end{equation}

Among the $(s,t)\in F$ choose $(x,y)$ such that $y$ maximises 
the  distance to~$\ell$. This means, that $st \in E(\Tsub{x}{y}')$
for any $(s,t)\in F\sm \{(x,y)\}$.
Consequently, $$|\Tsub{x}{y}'|=|E(\Tsub{x}{y}')|+1\geq |F|-1+1 =|F|.$$

Let   $r$ be the unique neighbour of the leaf $\ell$. Then $|\Tsub{\ell}{r}|=1$, and
we obtain
\begin{equation}\label{equ:max-r}
\pos{\newl-|\Tsub{\ell}{r}|} = \newl-1 \geq |\Tsub{x}{y}'| \geq |F|.
\end{equation}
\noindent
We conclude
\begin{align*}
\sum_{(s,t):st \in E(T)} \pos{\newl-|\Tsub{s}{t}|} 
&=
\pos{\newl-|\Tsub{\ell}{r}|}+ \pos{\newl-|\Tsub{r}{\ell}|} \\
& \qquad +\sum_{(s,t): st\in E(T') } \pos{\newl-|\Tsub{s}{t}|} \\
&\overset{\eqref{equ:max-r}}\geq
|F| + 0 +
\sum_{(s,t): st\in E(T')} \pos{\newl-|\Tsub{s}{t}|} \\
&\overset{\eqref{equ:[d-T'st]^+}}{=}
\sum_{(s,t):st \in E(T')} \pos{\newl-|\Tsub{s}{t}'|} \\
&\geq
\newl(\newl -1),
\end{align*}
where the last inequality follows by induction.
\end{proof}

\begin{proof}[Proof of Lemma~\ref{lem:baumlemma-2}]
	We proceed by induction on $|T|-\newl$.
	For the induction start, consider the case when $\newl=|T|$.
	Then
	\begin{align*}
		\pos{\newl-|\Tsub{s}{t}|} 
		  &= 
		   \pos{|T|-|\Tsub{s}{t}|} = |\Tsub{t}{s}|,
		\intertext{which yields}
		\sum_{(s,t)\in\delta^+(T^*)} \pos{\newl-|\Tsub{s}{t}|}  
		  &=
		 \sum_{(s,t)\in\delta^+(T^*)} |\Tsub{t}{s}|
          =  
         |T| - |T^*| =  \pos{\newl -|T^*|}.
	\end{align*}
	
Now assume $ |T| -\newl \geq 1$.
If every vertex in $T - V(T^*)$ is a leaf of $T$ then 
$t$ is a leaf for every $(s,t)\in\delta^+(T^*)$. This
implies 
$|\Tsub{s}{t}|  = |T| -1 \geq \newl$ and the left hand side of~\eqref{equ:baumlemma-2} vanishes.

Therefore we may assume that there is a leaf $\ell \notin T^*$ of $T$
such that neither $\ell$ nor its unique neighbour belongs to $V(T^*)$.
Set $T' = T - \ell$, and observe that, by choice of $\ell$, the set
$\delta^+(T^*)$ of edges leaving $T^*$ is the same in $T$ and in $T'$.
	Moreover, $|\Tsub{s}{t}|\geq|\Tsub{s}{t}'|$ holds for every $(s,t)\in\delta^+(T^*)$. 
	The desired inequality
	\begin{equation*}
		\sum_{(s,t)\in\delta^+(T^*)} \pos{\newl-|\Tsub{s}{t}|}
		\leq
		\sum_{(s,t)\in\delta^+(T^*)} \pos{\newl-|\Tsub{s}{t}'|}
   	 	\leq 
    	\pos{\newl -|T^*|}
 	\end{equation*}
	now follows by induction.
\end{proof}

\section{A lower bound on the maximum degree} \label{Sec:Proof_2k-1-bound}
Vizing~\cite{Vizing_russ} (see also Zhou et al.~\cite{Zhou_Na_Ni96})
proved that every graph of treewidth $k$ and
 maximum degree $\Delta \geq 2k$ has an edge-colouring with $\Delta$ colours. 
Proposition~\ref{prop:edge-colouring} shows that this bound is not tight.

A graph $G$ of maximum degree $\Delta$ is \emph{$\Delta$-critical}, 
if $\chi(G)=\Delta+1$ and all proper subgraphs can be edge-coloured using not more than $\Delta$ colours. 
For the proof of Proposition~\ref{prop:edge-colouring}
we use Vizing's adjacency lemma, as well as two adjacency lemmas
that involve the second neighbourhood.

\newtheorem*{vlem}{Vizing's adjacency lemma}
\begin{vlem}
Let $uv$ be an edge in a $\Delta$-critical graph. Then 
$v$ has at least $\Delta-\deg(u)+1$ neighbours of degree $\Delta$.
\end{vlem}

\begin{theorem}[Zhang~\cite{Zha00}]  % Thm 2
\label{thm:ajacency-3-path}
Let $G$ be a $\Delta$-critical graph, and let $uwv$ be  a path in $G$. 
If $\deg(u)+\deg(w)= \Delta +2$
then all neighbours of $v$ but $u$ and $w$ have degree~$\Delta$.
\end{theorem}

\begin{theorem}[Sanders and Zhao~\cite{SanZhao01}] %Lem 2.4
\label{lem:triangles}
Let $G$ be a $\Delta$-critical graph, and 
let $v$ be a common neighbour of $u$ and $w$ 
such that $\deg(u)  + \deg(v) + \deg(w) \leq 2\Delta +1$. 
Then there are at most $ \deg(u) + \deg(v) - \Delta -3 $ 
common neighbours 
$x \not= u $ of $v$ and $w$.
\end{theorem}

\newcommand{\q}{q}
\newcommand{\qv}{\q}

\newcommand{\p}{p}
\newcommand{\pw}{\p}

\newcommand{\vs}{v^*}
\newcommand{\ws}{w^*}
\newcommand{\us}{u^*}

\newcommand{\subtree}[1]{\lceil #1\rceil}

The rest of this subsection is dedicated to the proof of Proposition~\ref{prop:edge-colouring}.
To this end, let us assume Proposition~\ref{prop:edge-colouring} 
to be wrong. 
Then there is a  $\Delta$-critical graph $G$ of treewidth at most~$k$
for $\Delta=2k-1$.
(Note that the case $\Delta \geq 2k$ is covered by the above mentioned result of Vizing.)
Let $(T,\mathcal{B})$ be a smooth tree-decomposition of $G$
of width~$\leq k$. By picking an arbitrary root, we may consider $T$ as
a rooted tree. 
For any $s\in V(T)$, we denote by $\subtree s$ the subtree of $T$ rooted at~$s$, that is, the subtree of $T$ consisting of the vertices $t \in V(T)$ for which $s$ is contained in the path between $t$ and the root of $T$.

Recall the definition of $T(v)$ after Lemma~\ref{lem:baumlemma-2}.
Set $L=\{v \in V(G):\deg(v) \geq k+2\}$, and choose 
a vertex $\vs \in L$
that maximises the distance of $T(\vs)$ to the root (among the vertices in $L$). 
Let $\q$ be the vertex of $T(\vs)$ that achieves this distance. 
For $S \coloneqq N(q)\cap T(v^*)$ and any $s\in S$,
define $X_s = \bigcup_{t \in V(\subtree s)}B_{t}$, and let 
$X =  B_{\qv} \cup \bigcup_{s \in S} X_s$.
(See Figure~\ref{fig:edge-colouring} for an illustration.)
 Note that by the definition of $\vs$ and $q$
 \begin{equation}\label{equ:large-vertices-in-T_v}
  N(\vs) \subset X  \text{ and } X \cap L \subset B_{\q}. 
 \end{equation}
\begin{figure}
	\centering
	\begin{tikzpicture}[scale=1]
	\begin{scope}[shift={(0,0)},rotate=270]
	\tikzstyle{vertex}=[line width = 1pt, circle,draw,minimum size=15pt,inner sep=0pt, fill = white]
	\tikzstyle{edge-lin} = [draw,-,line width=2pt]
	%	\tikzstyle{weight} = [font=\tiny,draw,fill           = white,
	%	text           = black]

	% Draw graph vertices
	\foreach \pos/\name in { {(0,0)/1},
		{(1,1)/2},
		{(0,2)/3},
		{(2.3,1.8)/4}, 
		{(-1,3)/5}, 
		{(1.5,3)/6},
		{(-0.3,4)/7},
		{(-1.2,5)/8},
		{(1.5,6.5)/9},
		{(0,6)/10},
		{(-1,7)/11}}
	\node[vertex, align=center] (\name) at \pos { $v_{\name}$};

	% Draw copy of graph
	% Draw graph vertices
	\foreach \name in { 1,2,3,4,5,6,7,8,9,10,11}
	\node[vertex, align=center] (neu\name) at ($(\name) + (0,0)$) { $v_{\name}$};	
	% Draw graph edges
	\foreach \source/ \dest  in { neu1/neu2,neu1/neu3,neu2/neu3, neu2/neu4,neu2/neu6, neu3/neu5, neu3/neu6, neu4/neu6, neu5/neu7,neu6/neu7,neu6/neu10,neu7/neu8,neu7/neu9, neu7/neu10, neu8/neu10, neu8/neu11,neu9/neu10,
		neu10/neu11}
	\path[edge-lin] (\source) -- (\dest);

	% Draw tree decomposition
	
	\foreach \f/ \s / \t / \c / \name in {1/2/3/ForestGreen/1,
		2/3/6/WildStrawberry/2,
		2/4/6/CornflowerBlue/3,
		3/5/7/NavyBlue/4,
		3/6/7/Thistle/5,
		6/7/10/Peach/6,
		7/9/10/Magenta/7,
		7/8/10/JungleGreen/8,
		8/10/11/LimeGreen/9}			 
	{   \begin{scope}[transparency group]
		\begin{scope}[blend mode = soft light]
		\draw [rounded corners,fill = \c, fill opacity = 0.6, draw=none]
		(\f.center)--(\s.center)--(\t.center)--cycle;
		\end{scope}
		\end{scope}
		\coordinate (\f\s\t) at ($(barycentric cs:\f=.2,\s=.2,\t=.2) +(3,0)$);
		\node[vertex,fill=\c,align=center, fill opacity = 0.6] (t_\name) at (\f\s\t) {\textcolor{black}{}};
		\node[circle] (x\name) at ($(\f\s\t) + (-3,0)$) {$\textcolor{black}{ B_{\name}}$};
	}

	% Re draw graph vertices
	\foreach \name in { 1,2,3,4,5,6,7,8,9,10,11}
	\node[vertex, align=center] (neu\name) at ($(\name) + (0,0)$) { $v_{\name}$};
	\foreach \name in { 1,2,3,4,5,6,7,8,9}
	\node[vertex, align=center, fill = none] (bla\name) at ($(t_\name.center) + (0,0)$) { $t_{\name}$};	
	
	\foreach \source/ \dest  in {t_1/t_2, t_2/t_3, t_2/t_5, t_4/t_5, t_5/t_6, t_6/t_7, t_6/t_8, t_8/t_9}
	\path[edge-lin] (\source) -- (\dest);
	% Draw graph edges
	\foreach \source/ \dest  in { 1/2,1/3,2/3, 2/4,2/6, 3/5, 3/6, 4/6, 5/7,6/7,6/10,7/8,7/9, 7/10, 8/10, 8/11,9/10,
		{10}/11}
	\path[edge-lin] (\source) -- (\dest);
	
	\node[vertex,fill=\mycolour,minimum size=15pt] (asdf) at (10) {$v_{10}$};

	\foreach \source/ \dest  in {t_6/t_7, t_6/t_8, t_8/t_9}
	\path[edge-lin,line width=5pt,\mycolour] (\source) -- (\dest);
	
	% Arrows
	%	\draw [->,line width=2.5pt] (-2,4) -- node[above] {} (-1,4);
	%	\draw [<->,line width=2.5pt] (1.5,4) -- node[above] {} (2.5,4);
	
	% Labels
	%	\node[circle] (bla) at (-3,8) {\Large \text{Graph }$G$};
%	\node[circle] (bla) at (0.5,9) {\Large \text{Graph $G$ with bags  $\mathcal{B}$}};
	\node[circle] (bla) at (0.5,7) {\Large \text{$G$}};
	\node[circle] (bla) at (3.5,7) {\Large \text{} $T$};
	\end{scope}
	\end{tikzpicture}
	\caption{A graph $G$ with a smooth tree-decomposition $(T,\mathcal{B})$ of width 2.
		Here, $L = \{v_2,v_3,v_6,v_7,v_{10}\}$.
		If $T$ is rooted in $t_2$ then $v^*=v_{10}$,
$q=t_6$ and $B_q = \{v_7,v_9,v_{10}\}$.
		So, $T(v^*)$ consists of vertices $t_6,t_7,t_8,t_9$, and $S=\{t_7,t_8\}$ and  $X= \{v_7,v_8,v_9,v_{10},v_{11}\}$.
	}
	\label{fig:edge-colouring}
\end{figure}
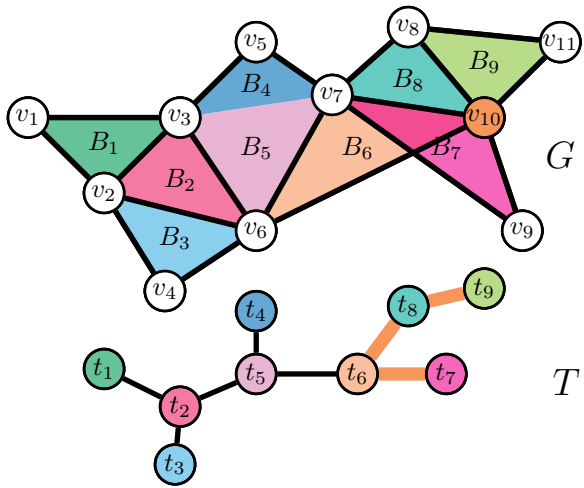
 \begin{claim} \label{cla:no-k+1-vertices}
All vertices of $X \setminus B_q$ have degree at most $k$. 
\end{claim}
 \begin{proof} [Proof of Claim~\ref{cla:no-k+1-vertices}]
  Suppose the statement to be false.
Then there is an $s\in S$ for which $X_s\setminus B_q$ 
contains a vertex of degree  at least  $k+1$.   
 Fix a vertex $\ws \in  \{ w \in X_s \setminus B_q:\deg(w) \geq k+1 \}=: L'$  that 
maximises the distance of $T(\ws)$ to $s$. 
Let $p$ be the vertex of $T(\ws)$ that achieves this distance. Set 
$Y =   \bigcup_{t \in V(\subtree \p)} B_t$. 
As in~\eqref{equ:large-vertices-in-T_v} we have $N(\ws) \subset Y$  and $  Y \cap L' \subset B_{\pw}$.  

By~\eqref{equ:large-vertices-in-T_v}, the vertex  $\ws \in X_s\setminus B_q$ has degree~$k+1$. Thus $\ws$ has a neighbour $\us$ outside $B_p$, which then has degree at most~$k$ (by choice of $\ws$).

Vizing's adjacency lemma implies that $\ws$ has at least 
$\Delta - \deg(u^*) +1 \geq 2k-1-k+1=k$  
neighbours of degree $\Delta$. By~\eqref{equ:large-vertices-in-T_v},
 all vertices of degree $\Delta$ of $Y$ have to be in $B_{\qv} \cap B_s$.  
Since by smoothness of the tree decomposition $B_{\qv} \cap B_s$ is a cutset of size at most~$k$, 
the vertex~$\ws$ is adjacent to all vertices in $B_{\qv} \cap B_s$. As $w^*$ is therefore adjacent to at most $k$ vertices of degree $\Delta$ it holds $\deg(u^*)=k$. 
By definition of $S$, the set $B_s$ contains $v^*$, which implies
that  $\vs$ is adjacent to $\ws$ and of degree~$\Delta$.
As $k\geq 4$, it follows that $\vs$ has degree~$\Delta=2k-1\geq k+3$, which means
by~\eqref{equ:large-vertices-in-T_v} that $\vs$ has at least three neighbours
of degree~$\leq k+1$. Thus, $\vs$ has a neighbour of degree~$\leq k+1$, which 
is neither $u^*$ nor $w^*$. This, however, 
 contradicts Theorem~\ref{thm:ajacency-3-path} (applied to $\vs,\ws,\us$).
 \end{proof}
 
 By~\eqref{equ:large-vertices-in-T_v} and since $v^*$ has degree at least~$k+2$,
the vertex $v^*$ has a neighbour $u\notin B_{\qv}$.
(In fact, $v^*$ has at least two such neighbours.)
 By Vizing's adjacency lemma, applied to $u\vs$, it follows that $\vs$ has
 at least $\Delta - \deg(u) +1 \geq k$
 neighbours of degree $\Delta$. 
{In particular, it follows from $X\cap L\subseteq B_q$, see~\eqref{equ:large-vertices-in-T_v},
that each of these neighbours lies in $B_q$. Since $|B_q|\leq k+1$ we get:}
 \begin{equation}\label{equ:Btw-all-delta-except-v}
\text{\em 
$v^*$ is adjacent to every vertex in $B_{\qv}$,
each of which has degree~$\Delta$. 
}
\end{equation}
{We also observe that $\Delta - \deg(u) +1 > k$ contradicts $|B_q|\leq k+1$, which means that
 \begin{equation}\label{degu}
\text{\em 
every $u\in N(\vs)\sm B_q$ has degree exactly~$k$.
}
\end{equation}
}
 
\begin{claim} \label{cla:neighbours-of-u}
Every $u\in N(\vs) \setminus B_{\qv}$ has exactly $k$ neighbours,
all of which are contained in $B_{\qv}$.
\end{claim} 

\begin{proof} [Proof of Claim~\ref{cla:neighbours-of-u}]
{By~\eqref{degu}, $u$ has exactly $k$ neighbours.
Since $B_{\qv}$ is a separator,} $u$ has all its neighbours in $X$. However, $u$ 
cannot be adjacent to any vertex $w$ of degree~$\leq k$;  otherwise we could extend any
$\Delta$-edge-colouring of $G-uw$ to $G$.
It follows from Claim~\ref{cla:no-k+1-vertices} that all of the $k$ neighbours of $u$ are in $B_{\qv}$.
\end{proof}

Since  the vertex $v^*$ has degree  at least $k+2$ and since $N(v^*)\subseteq X$, 
by~\eqref{equ:large-vertices-in-T_v}, 
it follows that $\vs$ has two neighbours $u,w$ that are contained in $X\sm B_q$.
By Claim~\ref{cla:neighbours-of-u}, the degree of $u$ and $w$ is~$k$. 
Thus, $\deg(u)+\deg(v^*)+\deg(w)\leq k+\Delta+k=2\Delta+1$.
Moreover, by Claim~\ref{cla:neighbours-of-u} and~\eqref{equ:Btw-all-delta-except-v}, 
the vertices $v^*$ and $w$ have
$k-1$ common neighbours in $B_{\qv}$. As $k-1>\deg(u)+\deg(v^*)-\Delta-3$,
we obtain a contradiction to Theorem~\ref{lem:triangles}.
This finishes the proof of Proposition~\ref{prop:edge-colouring}.

\section{Discussion} \label{Sec:discussion}

Proposition~\ref{prop:edge_bound} bounds the number of edges in a graph $G$ of fixed treewidth and maximum degree. 
A simpler bound~--~only considering the treewidth~--~is easily shown by induction~(see Rose~\cite{Rose74}):
\begin{equation}\label{eq:Rose}
2|E(G)|\leq 2 k |V(G)| - k (k+1)
\end{equation}
For $\Delta<2k$ and $|V(G)|>\Delta+1$ a straightforward computation shows that the bound of Proposition~\ref{prop:edge_bound} is strictly better than~\eqref{eq:Rose}.
The bounds are the same  if $\Delta=2k$ or if $|V(G)|=\Delta+1$. 
For $\Delta=2k$ this is illustrated by the $k$th power $P^k$ of a long path~$P$.

The bound in Proposition~\ref{prop:edge_bound}  is tight.
There are simple examples that show this: take the complete graph $K_k$ on $k$ vertices 
and add $ r \geq 1$ further vertices each adjacent to each vertex of $K_k$.
These graphs also demonstrate that Conjecture~\ref{conjecture} (if true) would be tight
or almost tight. Indeed, if $k+\lfloor \sqrt{k} \rfloor$ is even, and $k$ not a square,
then we obtain for $r=\lfloor \sqrt{k}\rfloor+1$ an overfull graph with maximum degree $\Delta=k+\lfloor \sqrt{k}\rfloor$. 
If  $k+ \lfloor \sqrt{k} \rfloor$ is odd, 
then, by setting $r= \lfloor \sqrt{k} \rfloor$, we obtain an overfull  graph with 
$\Delta = k+\lfloor \sqrt{k} \rfloor-1$.

These tight graphs, however, have a very special structure. 
In particular, they all satisfy $|V(G)|=\Delta(G)+1$. 
Both, Conjecture~\ref{conjecture} and Proposition~\ref{prop:edge_bound},
stay tight for an arbitrarily large number of vertices compared to $\Delta$:

\begin{proposition}\label{prop:tightexamples}
	For every $k_0\geq 4$ there is a $k \in \lbrace  k_0, k_0+1, \ldots, k_0 +8 \rbrace$ such that for every $n\geq 4k$
	there exists a graph $G$ on $n$ vertices with treewidth at most $k$
	and maximum degree $\Delta=k+\lfloor\sqrt k\rfloor< k+\sqrt k$ such that
	\[
	2|E(G)|=\Delta n-(\Delta-k)(\Delta-k+1).
	\]
	In particular, the graph $G$ is overfull whenever $n$ is odd.
\end{proposition}

We need the following lemma.

\begin{lemma} \label{lem:d_graphic}
Let $c,r \in \mathbb{N}$.
Then there is a graph with degree sequence
\begin{equation*}
\mathbf{d}=\big(\underbrace{c \ldots, c}_{r+1}, c-1, c-2, \ldots, 1\big) \in \Z^{c+r}
\end{equation*}
if and only if $4$ divides $c(2r+c+1)$ and if $r^2 \geq c$. 
\end{lemma}

We defer the proof of Lemma~\ref{lem:d_graphic} until the end of the section and only show sufficiency. A closer look at the arguments in the proof yields necessity.

\begin{proof}[Proof of Proposition~\ref{prop:tightexamples}]

We start by showing with a case distinction that there is a $k  \in  \{  k_0, k_0 +1, \ldots, k_0+ 8\}$ such that
\begin{equation}
k\equiv \left\lfloor \sqrt{ k} \right\rfloor \pmod 8 \text{ and } \left\lfloor \sqrt{ k} \right\rfloor <  \sqrt{ k} . \label{eq:mod8}
\end{equation}
To this end, let $i$ such that $ \lfloor \sqrt{k_0} \rfloor \equiv  k_0 + i \pmod 8$ and $0 \leq i \leq 7$. 

Firstly, let us assume that $i = 0$.
 If $k_0$ is not a square, then $k=k_0$ satisfies~\eqref{eq:mod8}.
 Otherwise $k=k_0+8$ satisfies \eqref{eq:mod8} as $k_0\geq 4>1$, and consequently $\lfloor \sqrt{k_0+8} \rfloor=\sqrt{k_0}$.

Secondly, we consider the case that $i \neq 0$. 
 If $\lfloor \sqrt{k_0+i} \rfloor = \lfloor \sqrt{k_0} \rfloor $, then $k=k_0+i$ satisfies $\lfloor \sqrt{k_0}\rfloor \equiv k \pmod 8$ and $\sqrt{k} > \sqrt{k_0} \geq \lfloor \sqrt{k_0} \rfloor = \lfloor \sqrt{k} \rfloor $, which shows~\eqref{eq:mod8}.
 If, on the other hand, $\lfloor \sqrt{k_0+i} \rfloor > \left\lfloor \sqrt{k_0} \right\rfloor $, then $\lfloor \sqrt{k_0+i} \rfloor = \lfloor \sqrt{k_0} \rfloor +1= \lfloor \sqrt{k_0+i+1} \rfloor$ as $k_0 \geq 4$. Set  $k = k_0 + i +1$. By choice of $i$, we have  $\lfloor \sqrt{k_0} +1 \rfloor  \equiv k \pmod 8$. Thus, we obtain $\lfloor \sqrt{k} \rfloor  \equiv k \pmod 8$ as desired. Moreover, $ \sqrt{k} > \sqrt{k_0+i} \geq \lfloor \sqrt{k_0+i} \rfloor = \lfloor \sqrt{k_0} \rfloor +1= \lfloor \sqrt{k} \rfloor.$

In all cases an element of $ \{ k_0, k_0 +1, \ldots, k_0+ 8 \}$ satisfies \eqref{eq:mod8}. 

Next  we show that for any $n \geq 4k$, there is a graph $G$ of treewidth $k$ whose degree sequence 
$\left( \deg_G(v_1), \deg_G(v_2), \ldots, \deg_G(v_n)  \right)$ equals
\begin{equation}\label{target}
(k, k+1, \ldots, \Delta-1, \Delta, \ldots, \Delta, \Delta-1, \ldots, k+1, k)
\end{equation}
with $\Delta= k + \lfloor \sqrt{ k} \rfloor $.
A computation similar to Lemma~\ref{lem:not_overfull} shows that $G$ is overfull if $|V(G)|$ is odd.

 We construct $G$ in three steps. First we take a power of a path, where all but the outer vertices 
have the right degree. We increase the degree of the outer vertices by connecting them to vertices towards 
the middle of the path. This will create some degree excess for the used vertices. 
We balance this by deleting a subgraph $H$ provided by Lemma~\ref{lem:d_graphic}. The construction is illustrated 
in Figure~\ref{constructfig}. Note that for ease of exposition the parameters $k$ and $\Delta$ are not as in this proof.

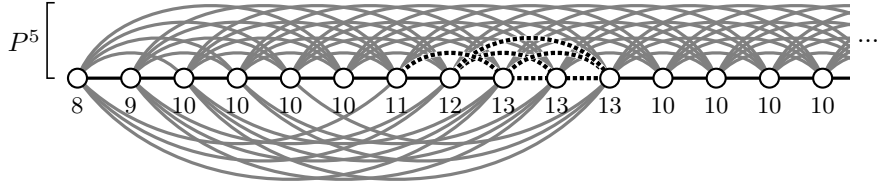
\begin{figure}[htb]
	\centering
	\begin{tikzpicture}[node distance=0mm and 0mm]
	
	\def\xstep{0.7}
	\def\vxnumber{15}
	
	\tikzstyle{gedge}=[hedge,gray]
	
	\draw[semithick] (0.4,0) -- ++(-0.1,0) -- ++(0,1) -- ++(0.1,0);
	\node at (0,0.5) {$P^5$};
	\node at (\vxnumber*\xstep+0.6,0.5) {$...$};
	
	%\draw[color=blue] (0.5,-1.5) rectangle (\vxnumber*\xstep+0.5*\xstep,1.5);
	\clip (0.5,-1.5) rectangle (\vxnumber*\xstep+0.5*\xstep,1.5);
	
	\foreach \i in {1,...,\vxnumber}{
		\draw[hedge] (\i*\xstep,0) -- (\i*\xstep+\xstep,0);
		\draw[gedge,bend left=55] (\i*\xstep,0) to (\i*\xstep+2*\xstep,0);
		\draw[gedge,bend left=60] (\i*\xstep,0) to (\i*\xstep+3*\xstep,0);
		\draw[gedge,bend left=65] (\i*\xstep,0) to (\i*\xstep+4*\xstep,0);
		\draw[gedge,bend left=70] (\i*\xstep,0) to (\i*\xstep+5*\xstep,0);
	}
	
	\foreach \i in {1,2,3,4,5}{
		\draw[gedge,bend right=45] (\i*\xstep,0) to (\i*\xstep+6*\xstep,0);
	}
	\foreach \i in {1,2,3,4}{
		\draw[gedge,bend right=50] (\i*\xstep,0) to (\i*\xstep+7*\xstep,0);
	}
	\foreach \i in {1,2,3}{
		\draw[gedge,bend right=55] (\i*\xstep,0) to (\i*\xstep+8*\xstep,0);
	}
	
	\draw[medge] (9*\xstep,0) -- (10*\xstep,0) -- (11*\xstep,0);
	\draw[medge,bend left=55] (8*\xstep,0) to (10*\xstep,0); 
	\draw[medge,bend left=55] (9*\xstep,0) to (11*\xstep,0); 
	\draw[medge,bend left=60] (8*\xstep,0) to (11*\xstep,0); 
	\draw[medge,bend left=55] (7*\xstep,0) to (9*\xstep,0);

	\foreach \i in {1,...,\vxnumber}{
		\node[hvertex] (v\i) at (\i*\xstep,0){};
	}

	\foreach \i in {3,4,5,6,12,13,14,15}{
		\node[below=of v\i] {\small $10$};
	}
	\node[below=of v1] {\small $8$};
	\node[below=of v2] {\small $9$};
	\node[below=of v7] {\small $11$};
	\node[below=of v8] {\small $12$};
	\node[below=of v9] {\small $13$};
	\node[below=of v10] {\small $13$};
	\node[below=of v11] {\small $13$};

	\end{tikzpicture}
	\caption{Extreme example for $k=8$ and $\Delta=10$. The graph $H$ is dotted.}\label{constructfig}
\end{figure}

Let $P$ be a $\Delta/{2}$-th power of a path on vertices $v_1, \ldots, v_n$. This means, $v_i$ and $v_j$ are adjacent if and only if  $0 < |i - j| \leq \Delta/{2}$. As $P$ is symmetric, and as $G$ will be symmetric as well, we concentrate on the part of $P$ on the vertices $v_1, \ldots, v_{\lceil {n}/{2} \rceil}$. We tacitly agree that any additions and deletions of edges are also applied to the other half of $P$.

Comparing the degrees of $P$ to~\eqref{target} we see that all vertices have the target degree except for the initial 
vertices $v_1, \ldots, v_{\Delta / 2}$, whose degree is too small. For $i=1,\ldots, \Delta-k$ the vertex $v_i$ has 
degree~$\Delta/2-1+i$ but should have degree~$k-1+i$. We fix this by connecting $v_i$ to $v_{i+\Delta/2+1},\ldots, v_{i+k+1}$. 
For $i=\Delta-k+1,\ldots,\Delta/2$, the vertex $v_i$ should have degree~$\Delta$ but has 
degree~$\Delta/2-1+i$. We make $v_i$ adjacent to each of $v_{i+\Delta/2+1},\ldots,v_{\Delta+1}$.

Denote the obtained graph by $P'$ and observe that its vertices in the range of $1,\ldots , \left\lceil n / 2\right\rceil $ 
have the following degrees
\begin{equation*}
\underbrace{k, k+1, \ldots, \Delta}_{1, \ldots, \Delta -k +1}, 
\underbrace{\Delta, \ldots, \Delta}_{ \Delta - k+2, \ldots, \Delta / 2+1}, 
\underbrace{\Delta+1, \ldots, k + \tfrac{\Delta}{2}}_{\tfrac{\Delta}{2} +2, \ldots, k+1},
\underbrace{ k + \tfrac{\Delta}{2}, \ldots, k+\tfrac{\Delta}{2}}_{k+2, \ldots, \Delta +1}, 
\underbrace{ \Delta, \ldots, \Delta}_{\Delta +2, \ldots, \left\lceil n / 2\right\rceil}
\end{equation*}

Hence all but the vertices $v_i$ with index $i$ between $\Delta/2 +2$ and $\Delta+1$ have the correct degree. 
The difference between their degree in $P'$ and the desired degree is 
\begin{equation} \label{SequenceDiffOfDegrees}
\mathbf{d} = \big( 1,2,\ldots,k-\tfrac{\Delta}{2}-1 ,\underbrace{k-\tfrac{\Delta}{2} ,\ldots, k - \tfrac{\Delta}{2} }_{\Delta -k +1} \big).
\end{equation}
Set $c = k- \frac{\Delta}{2}= \tfrac{1}{2} \left( k- \lfloor \sqrt{k} \rfloor \right)$ and $r = \Delta -k$. 
Note that $k$ is chosen in such a way (see \eqref{eq:mod8}) that $c$ is divisible by~$4$. As furthermore  
$r^2=(\Delta-k)^2 = \lfloor \sqrt{k} \rfloor^2 \geq \frac{1}{2} \left(k- \lfloor \sqrt{k}\rfloor\right)=c
$,  Lemma~\ref{lem:d_graphic} yields that there is a graph $H$ with degree sequence $\mathbf{d}$. 
Since the vertices  $v_{\Delta/2 +2}, \ldots, v_{\Delta +1}$ induce a complete graph in $P'$ there is a copy of $H$ in $P'$, 
such that deleting its edges results in a graph $G$ of the desired degree sequence. Note that for any two adjacent 
vertices $v_i$, $v_j$ in $P'$ it holds that $|i -j| \leq k$. This implies that $P'$ is a subgraph of a $k$-th power of a path. Thus the subgraph $G$ of $ P'$ has treewidth at most $k$. This finishes the proof.
\end{proof}

To prove Lemma~\ref{lem:d_graphic} we use the Erd\H os-Gallai-criterion:
\begin{theorem}[Erd\H os and Gallai~\cite{EG}]\label{thm:erdos-gallai}
	There is a graph with degree sequence $d_1 \geq \cdots \geq d_{n}$
	if and only if $\sum_{i=1}^nd_i$ is even and if for all $\ell=1, \ldots, {n}$
	\begin{equation}\label{eq:EG}
	\sum\limits_{i=1}^\ell d_i 
	\leq 
	\ell(\ell-1) + \sum\limits_{i=\ell+1}^{n} \min(d_i, \ell).
	\end{equation}
\end{theorem}

\begin{proof}[Proof of Lemma~\ref{lem:d_graphic}]
	We check the conditions of Theorem~\ref{thm:erdos-gallai} for the degree sequence $\mathbf{d}$. The parity condition holds as $4$ divides $c(2r+c+1)$ and
	$$\sum_{i=1}^{c+r}d_i=cr+\frac{c(c+1)}{2} = \frac{c}{2}(2r+c+1).$$
	
	Let of us now verify~\eqref{eq:EG}. If $\ell > c$, then
	\[
	\sum_{i=1}^\ell d_i\leq c\ell \leq \ell(\ell-1) \leq 
	\ell(\ell-1) + \sum_{i=\ell+1}^{c+r} \min(d_i, \ell).
	\]
	Thus we can assume that $\ell \leq c$. Two remarks: Firstly,  $\min(d_i, \ell)=\ell$ for $i=1, \ldots, \leq c+r-\ell+1$. Consequently, if $2 \ell \leq c+r$ then
	\begin{align}\notag
	\ell(\ell-1) +\sum_{i=\ell+1}^{c+r} \min(d_i, \ell) 
	&=\ell(\ell-1)+(c+r-2\ell+1)\ell + \frac{\ell(\ell-1)}{2}\\
	&=\tfrac{\ell}{2}(2r-1-\ell)+c\ell. \label{RHS}
	\end{align}
	Secondly,  if $\ell > r$, then
	\begin{equation}\label{blurb}
	\sum_{i=1}^\ell d_i  = c\ell-\frac{(\ell-r-1)(\ell-r)}{2}
	=c\ell+\tfrac{\ell}{2}(2r+1-\ell)-\tfrac{1}{2}(r^2+r).
	\end{equation}

	Now suppose that $2\ell\leq c+r$. For $\ell\leq r$, we have $\sum_{i=1}^\ell d_i=c\ell$ and hence~\eqref{eq:EG} is easily seen to be satisfied in light of~\eqref{RHS}. On the other hand, for $\ell > r$ the assumption of $r^2 \geq c$ together with a comparison of~\eqref{RHS} and~\eqref{blurb} gives~\eqref{eq:EG}.

	So let $2\ell>c+r$. This implies that  $\ell > r$. Consequently, the right hand side of~\eqref{eq:EG} is
	\begin{align*}
	\ell(\ell-1)+\sum_{i=\ell+1}^{c+r} \min(d_i, \ell)
	&=\ell(\ell-1)+\sum_{i=\ell+1}^{c+r} d_i\\
	&=\ell(\ell-1) + \tfrac{1}{2}(c+r-\ell)(c+r-\ell+1).
	\end{align*}
It follows from equation~\eqref{blurb}  that \eqref{eq:EG} is satisfied if the following expression is non-negative. 
\begin{align}
&2\ell(\ell-1) + (c+r-\ell)(c+r-\ell+1) -(2 c\ell+ \ell(2r+1-\ell)-(r^2+r)) \nonumber\\
&= 4\ell^2   -4 \ell(c+r)  +   (c+r)^2  +(c+2r+r^2)  - 4 \ell \nonumber\\
&= (2 \ell - (c+r))^2  +(c+r) +(r+r^2)  - 4 \ell \nonumber\\
&= (2 \ell - (c+r))^2  -2\Big(2\ell - \frac{(c+r) +(r+r^2)}{2}\Big)    \label{eq:calc}
\end{align}
First, let $r^2=c$. Then \eqref{eq:calc} equals 
\begin{equation}
(2 \ell - (c+r))^2  -2 (2 \ell- (c+r)) \label{eq:f-function}
\end{equation}
The term~\eqref{eq:f-function} is negative only if $2 \ell -(c+r)=1$. As $c+r=r^2+r$ is even (for any integer $r$), \eqref{eq:f-function}  and thus \eqref{eq:calc} is non-negative. 

Now let $r^2 > c$. Then \eqref{eq:calc} is strictly greater than \eqref{eq:f-function} and hence non-negative.
This shows that \eqref{eq:EG} is satisfied. 

As \eqref{eq:EG} holds for all $\ell$, there is a graph with degree sequence $\mathbf{d}$.
\end{proof}

\section{Degenerate graphs}
\label{Sec:degenerate}
Recall that a graph $G$ is \emph{$k$-degenerate} if there is an enumeration $v_n,\ldots, v_1$
of the vertices such that $v_{i-1}$ has degree at most~$k$
in $G-\{v_n,\ldots, v_i\}$ for every~$i$. 
By simple induction following the elimination order (or recalling the formula $1+2+\ldots +N=N(N+1)/2$), we can obtain a bound with half the degree loss of \eqref{eq:edge_bound}:
\begin{equation} \label{eq:kdeg_edge-bound}
2 |E(G)| \leq \Delta |V(G)| - \tfrac{1}{2}(\Delta - k)(\Delta - k +1).
\end{equation}

\begin{figure}
	\begin{center}
		\begin{tikzpicture}
		
		\tikzstyle{hugenode}=[draw,circle,line width=3pt,color=gray,minimum size=2cm]
		\tikzstyle{hugeedge}=[line width=10pt,color=gray]
		\tikzstyle{smallvx}=[thick,circle,inner sep=0.cm, minimum size=2.5mm, fill=white, draw=black]
		\tikzstyle{smalledge}=[very thick]
		
		\def\hugeradius{2cm}
		\def\smallradius{0.4cm}
		
		\def\smallcomplete#1{
			\pgfmathsetmacro{\angle}{360/(#1+1)}
			\foreach \k in {0,...,#1}{
				\node[smallvx] (s\k) at (\k*\angle+90:\smallradius){};
			}
			\foreach \k in {0,...,#1}{
				\foreach \l in {\k,...,#1}{
					\draw[smalledge] (s\k) -- (s\l);
				}
			}
		}
		
		\foreach \i in {0,1,2,3,4}{
			\node[hugenode] (H\i) at (\i*72+90:\hugeradius){};
			\node[smallvx,fill=black] (u\i) at (\i*72+90:1.3*\hugeradius){};
		}
		\foreach \i in {1,2,3,4}{
			\begin{scope}[shift={(\i*72+90:0.9*\hugeradius)},rotate=\i*72]
			\smallcomplete{\i}
			\end{scope}
		}
		\node[smallvx]  at (90:0.9*\hugeradius){};
		
		\foreach \i in {0,1,2,3,4}{
			\foreach \j in {\i,...,4}{
				\draw[hugeedge] (H\i) -- (H\j);
			}
		}
		
		\foreach \i in {0,1,2,3}{
			\pgfmathtruncatemacro{\prev}{\i+1}
			\draw[very thick,dashed,-latex,bend left=30,shorten <=2pt,shorten >=2pt] (u\prev) to (u\i);
		}
		
		\end{tikzpicture}
	\end{center}
	\caption{The graph $G_5$ with the vertices $v_i$ drawn in black; thick gray edges
		indicate that two vertex sets are complete to each other; elimination order of the $v_i$
		is shown in dashed lines}
	\label{fig:G5}
\end{figure}
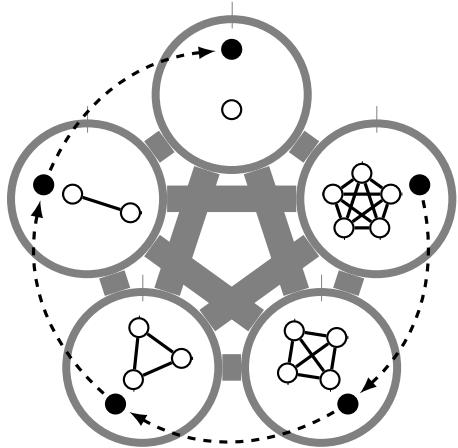

The bound in \eqref{eq:kdeg_edge-bound} turns out to be tight for some $\Delta, k$ as the construction below shows. Moreover, by~\eqref{eq:kdeg_edge-bound}, Theorem~\ref{thm:fractional-tw-maxdeg} can easily be transferred: 
any simple $k$-degenerate graph of maximum degree $\Delta\geq k + 1/2 + \sqrt{2k+1/4} $  is not overfull and therefore has fractional chromatic index $\chif(G)=\Delta$.

Consider a positive integer $p$ and let $G_p$ be the complement of the disjoint union of $p$
stars $K_{1,1},K_{1,2},\ldots, K_{1,p}$;
see Figure~\ref{fig:G5}. Denote the centre of the $i$th star by $v_i$, and 
let $W$ be the union of all leaves. The graph $G_p$ has $n= p(p+1)/2 +p$ vertices and satisfies
$\deg(v_i) = n-1-i $ for $ i=1, \ldots , p$ and $\deg(w) = n-2 $ for $w \in W$. In particular, the maximum degree of $G_p$ is
$\Delta =  n-2$.
Setting $k= n-1-p$, we note that $G_p$ is $k$-degenerate
as $v_p,v_{p-1},\ldots,v_1$ followed by an arbitrary enumeration 
of $W$ is an elimination order. 
Finally, we observe that $G_p$ satisfies~\eqref{eq:kdeg_edge-bound}
with equality.

\bibliographystyle{amsplain}

\bibliography{bibliography}

\vfill

\small
\vskip2mm plus 1fill
\noindent
Version \today{}
\bigbreak

\bigbreak

\noindent
\begin{tabular}{cc}
\begin{minipage}[t]{0.5\linewidth}
Henning Bruhn\\{\tt <henning.bruhn@uni-ulm.de>}\\
Laura Gellert\\
{\tt <laura@lk-gellert.de>}\\
Universit\"at Ulm, Germany
\end{minipage}
&
\begin{minipage}[t]{0.5\linewidth}
Richard Lang\\ 
{\tt <r.lang.1@bham.ac.uk>}\\
University of Birmingham, UK %\\ 
\end{minipage}
\end{tabular}

\end{document}